\documentclass[11pt]{article}
\usepackage{amssymb,amsmath,amsthm}
\usepackage{enumerate}
\usepackage{tikz}
\usetikzlibrary{arrows}

\newenvironment{packedItem}{
\begin{itemize}
  \setlength{\itemsep}{1pt}
  \setlength{\parskip}{0pt}
  \setlength{\parsep}{0pt}
}{\end{itemize}}

\let\oldmarginpar\marginpar
\renewcommand\marginpar[1]{\-\oldmarginpar[\raggedleft\footnotesize #1]%
{\raggedright\footnotesize #1}}

\newtheorem{theorem}{Theorem}
\newtheorem{corollary}[theorem]{Corollary}
\newtheorem{lemma}[theorem]{Lemma}
\newtheorem{proposition}[theorem]{Proposition}

\newtheorem{question}[theorem]{Question}

\newtheorem{example}[theorem]{Example}

\newcommand{\av}{{\bf a}}
\newcommand{\bv}{{\bf b}}

\newcommand{\cv}{{\bf c}}

\newcommand{\xv}{{\bf x}}

\newcommand{\wv}{{\bf w}}

\newcommand{\st}{{\text{st}}}

\newcommand{\vanish}[1]{}

\begin{document}

 \title{Enumerating cycles in the graph of overlapping permutations}

 \author{
 John Asplund\\
 {\small Department of Technology and Mathematics} \\
 {Dalton State College} \\
 {\small Dalton, GA 30720, USA} \\
 {\small jasplund@daltonstate.edu}\\
 \\
 N. Bradley Fox\\
 {\small Department of Mathematics and Statistics} \\
 {Austin Peay State University} \\
 {\small Clarksville, TN 37044, USA} \\
 {\small foxb@apsu.edu}
  }
 \date{}
 \maketitle

%
%
%
%
%
%

\begin{abstract}
The graph of overlapping permutations is a directed graph that is an analogue to the De Bruijn graph.  It consists of vertices that are permutations of length $n$ and edges that are permutations of length $n+1$ in which an edge $a_1\cdots a_{n+1}$ would connect the standardization of $a_1\cdots a_n$ to the standardization of $a_2\cdots a_{n+1}$.  We examine properties of this graph to determine where directed cycles can exist, to count the number of directed $2$-cycles within the graph, and to enumerate the vertices that are contained within closed walks and directed cycles of more general lengths.
\end{abstract}



%
%

\section{Introduction}

In this paper we will discuss an analogue to a classical object in combinatorics, De Bruijn graphs. The graphs in this paper are all directed graphs with arcs (edges with an orientation) between vertices. For a set $\{0,1,\ldots,q-1\}$, let $\{0,1,\ldots,q-1\}^n$ be the set of all strings of length $n$ with elements from $\{0,1,\ldots,q-1\}$. A De Bruijn graph has the vertex set $\{0,1,\ldots,q-1\}^n$ and a directed edge from each vertex $x_1x_2\cdots x_{n}$ to the vertex $x_2x_3\cdots x_{n+1}$. In other words, there is an edge from $\av$ to $\bv$ if and only if the last $n-1$ elements of~$\av$ and the first $n-1$ elements of $\bv$ are the same. One of the properties of De Bruijn graphs with vertex set $\{0,1,\ldots,q-1\}^n$ that has been studied is the number of directed cycles of length $k$, for $k\leq n$, which is shown in~\cite{golomb}. For ease of notation, we will call a directed cycle of length~$k$ a $k$-cycle.

We will examine the number of $k$-cycles in the graph of overlapping permutations, $G(n)$, which was introduced in~\cite{CDG} and was studied in~\cite{Kitaev1} and Chapter 5 of~\cite{Kitaev2}.  The existence of overlapping cycles on a similar graph can also be seen in~\cite{HoranHurlbert}. For a permutation $\av=a_1a_2\cdots a_n$, we denote the \textit{standardization} of the substring $a_{s+1}a_{s+2}\cdots a_{s+t}$ by $\st(a_{s+1}a_{s+2}\cdots a_{s+t})=b_1b_2\cdots b_{t}$ where $b_i\in \{1,\ldots,t\}$ with $b_i<b_j$ if and only if $a_{s+i}< a_{s+j}$ for all $1\leq i,j\leq t$. Let $G(n)$ be the graph with vertex set as the set of all permutations of length $n$ and consisting of a directed edge from each vertex $\av=a_1a_2\cdots a_n$ to $\bv=b_1b_2\cdots b_n$ if $\st(a_2a_3\cdots a_n)=\st(b_1b_2\cdots b_{n-1})$. This adjacency condition is why this is called a graph of overlapping permutations and is analogous to the overlapping condition of strings in the De Bruijn graph. Be aware that $G(n)$ can have multiple directed edges between the same vertices and in the same direction.

Enumerating the cycles in the graph of overlapping permutations was first attempted by Ehrenborg, Kitaev, and Steingr\'\i msson in~\cite{EKS}.  The authors focused on the graph $G(n,312)$, which is the graph of overlapping permutations where the vertices and edges are permutations which avoid the pattern $312$, i.e., a subsequence of three entries in the permutations of the vertices and edges whose standardization is $312$.  They determined the number of closed walks of length $k$ for $k\leq n$, as well as the number of $k$-cycles.
See Theorems 5.1 and 5.2 in~\cite{EKS} for these results.  An important observation for these enumerations is that they do not depend upon $n$.  This is not the case for $G(n)$, as increasing the length of the permutations will increase the number of $k$-cycles when looking at the entire graph $G(n)$ instead of a subgraph that avoids certain patterns.

In the following section, we will discuss some pertinent graph theory terminology and present an example of $G(3)$. In Section~\ref{properties}, we will introduce conditions on which a closed walk exists at a particular vertex, as well as discuss the existence of two edges between the same two vertices and multiple closed walks stemming from a single vertex.  In Section~\ref{2cycles}, we count the $2$-cycles in $G(n)$.  Section 5 includes enumerations for the number of vertices contained within closed walks and a correspondence between the number of closed $k$-walks in $G(n)$ and the number of $k$-cycles in $G(n)$ as long as $k$ is prime.  Finally, Section 6 consists of further research questions involving cycles within the graph of overlapping permutations.


\section{Preliminaries and an Example}

Before we introduce our results, we must first establish some graph theory terminology.  The following definitions are for directed graphs.  We consider a sequence $(v_1, e_1, v_2, e_2, v_3,\ldots, v_{k},e_k, v_{k+1})$  of vertices $v_i$ and edges $e_i=(v_i,v_{i+1})$ to be a \textit{walk} of length $k$, or a $k$-walk.  We will only list the vertices in the walk for simplification within our proofs, although we do consider two walks (and two cycles) to be distinct if the sequence of vertices are the same, but edges differ due to pairs of vertices being connected by multiple edges.  If $v_1=v_{k+1}$, the walk is called a \textit{closed walk} and will usually be written as the sequence $(v_1, v_2,\ldots, v_k)$ to represent the closed $k$-walk.  If the vertices are all distinct, the closed $k$-walk is a \textit{$k$-cycle}. For more information about graph terminology not mentioned in this article, see~\cite{west}.

When counting closed walks, we consider the closed walk $(v_1, e_1, v_2, e_2, v_3,\ldots, v_{k},e_k, v_{1})$ to be the same as the closed walk $(v_j, e_j, v_{j+1}, e_{j+1}, v_{j+2},\ldots, v_{k},e_k, v_{1},e_1,v_2,\ldots, v_{j-1},e_{j-1},v_j)$ that has a different starting vertex.  Likewise, they are counted as one cycle if the vertices are distinct.  More formally, we are attempting to count equivalence classes of closed walks (and cycles) by considering shifted sequences of vertices and edges within a closed walk to be equivalent.

Given a vertex $\av=a_1 a_2\cdots a_n$, we consider two important vertices related to it.  First, the \textit{complement} of $\av$ is the vertex $\overline{\av}=(n+1-a_1) (n+1-a_2)\cdots (n+1-a_n)$.  Second, we define the \textit{cyclic shift} of the vertex $\av$ by $\sigma(\av)=a_2 a_3\cdots a_n a_1$.

Figure 1 displays the graph $G(3)$ with permutations of length $3$ as the six vertices.  
The twenty-four edges, although unlabeled in the figure, correspond to the permutations of 
length $4$.  Recall the definition of the edges in $G(n)$ as $c_1\cdots c_{n+1}$ connecting 
the standardization $\av=a_1\cdots a_n=\st(c_1\cdots c_n)$ to the standardization 
$\bv=b_1\cdots b_n=\st(c_2\cdots c_{n+1})$.  This directly implies there is an edge from $\av$ to~$\bv$ if and only if $\st(a_2\cdots a_n)=\st(b_1\cdots b_{n-1})$.  Hence we see in this example 
that there exists an edge from $123$ to $132$ since $\st(23)=\st(13)=12$, whereas there is 
no returning edge from $132$ to $123$ since $\st(32)=21\neq \st(12)=12$.

The number of cycles in the graph are as follows: two $1$-cycles at the trivial 
vertices $123$ and $321$, six $2$-cycles, and twenty-six $3$-cycles.  Observe that with 
$2$-cycles some pairs of vertices create two $2$-cycles because there are multiple edges between the same two vertices.
An example is the cycle $(132, 213)$ made from either the edge pairs $1324$ and $2143$ or the pair 
$1324$ and $3142$.  As many as eight $3$-cycles can be made from a triple of vertices, which 
occurs with the $3$-cycle $(132, 321, 213)$, which can make the enumeration of long cycles 
rather difficult.  Additionally, note that vertices can be included in multiple cycles with differing 
vertices.  For instance, the vertex $231$ is contained in the $2$-cycles $(231, 312)$ and 
$(231,213)$.  A final observation is that the trivial vertices are within $k$-cycles for each 
$1\leq k\leq 6$ except for $k=2$.  The fact that they are not in $2$-cycles will be generalized in the next section.


\begin{figure}\label{G3}
\begin{center}
\includegraphics[scale=1]{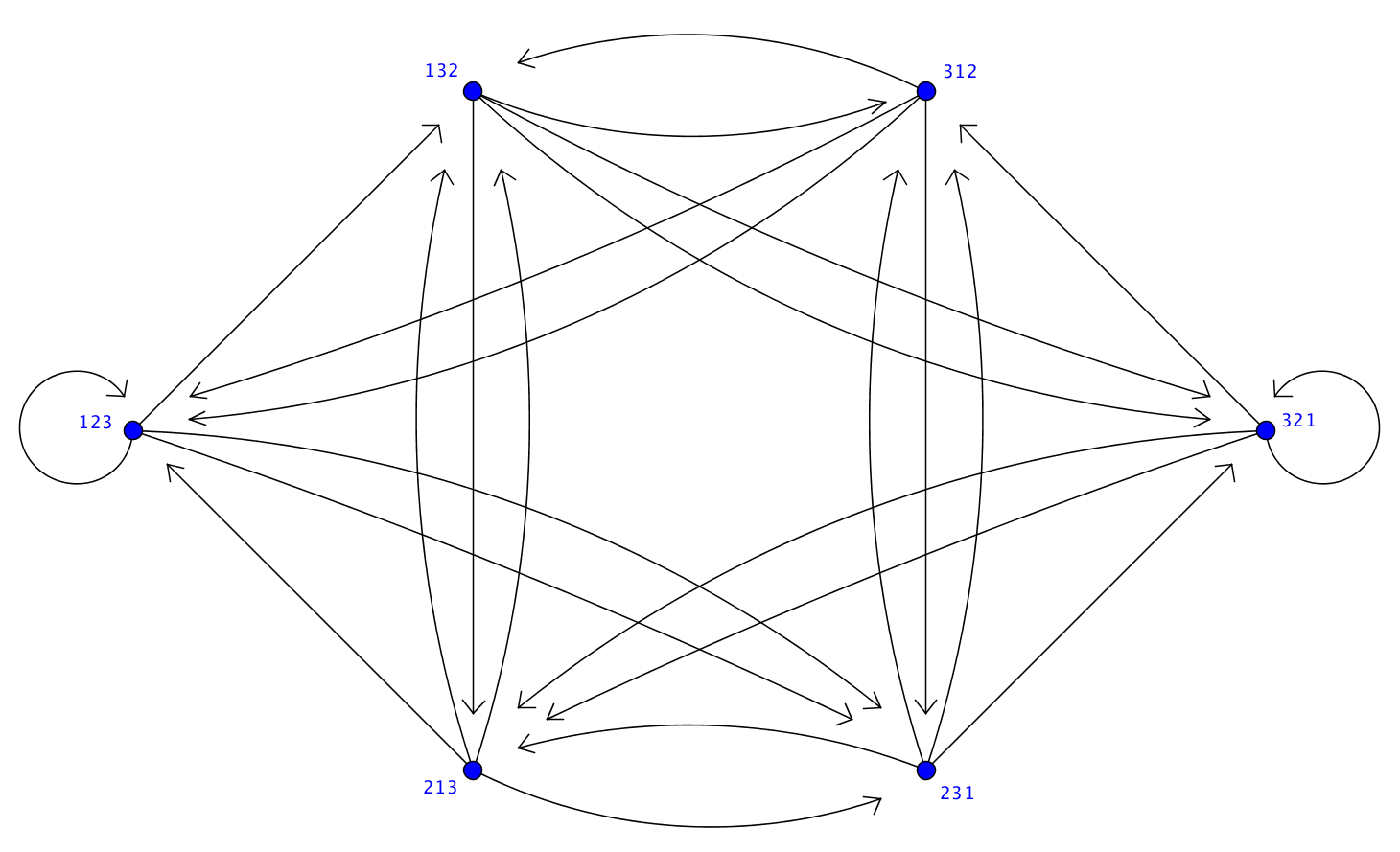}
\caption{The graph of overlapping permutations of length $3$, $G(3)$}
\end{center}
\end{figure}

\section{Properties of Closed Walks and Cycles in $G(n)$}\label{properties}

We will now introduce properties regarding the inclusion of vertices within closed walks and cycles.  
Unless otherwise noted, assume $k$ is an integer with $2\leq k\leq n-1$.
We first establish a necessary condition for the existence of a closed $k$-walk through a vertex. 
\begin{theorem}\label{cycleNecessary}
If a vertex $\av=a_1a_2\cdots a_n$ is in some closed $k$-walk in $G(n)$, then 
\[
\st(a_1a_2\cdots a_{n-k})=\st(a_{k+1}a_{k+2}\cdots a_n).
\]
\end{theorem}
\begin{proof}
Assume $\av=a_1a_2\cdots a_n$ is a vertex in the closed $k$-walk $\left(\av, \av^{(2)}, \ldots, \av^{(k)}\right)$ where $\av^{(i)}=a^{(i)}_1a^{(i)}_2\cdots a^{(i)}_n$.  The directed edge from $\av$ to $\av^{(2)}$ in the walk above gives us the equality
$\st(a_2\cdots a_n)=\st\left(a^{(2)}_1\cdots a^{(2)}_{n-1}\right)$, which when narrowing down to the last $n-k$ elements of $\av$ becomes $$\st(a_{k+1}\cdots a_n)=\st\left(a^{(2)}_k\cdots a^{(2)}_{n-1}\right).$$  Then we use the second directed edge in this walk, which provides the equality
$\st\left(a^{(2)}_2\cdots a^{(2)}_{n}\right)=\st\left(a^{(3)}_1\cdots a^{(3)}_{n-1}\right)$. By combining this with the previous equality we attain $$\st(a_{k+1}\cdots a_n)=\st\left(a^{(3)}_{k-1}\cdots a^{(3)}_{n-2}\right).$$  Continuing this reasoning for the subsequent edges leading up to the vertex $\av^{(k)}$, we have 
$$\st(a_{k+1}\cdots a_n)=\st\left(a^{(k)}_2\cdots a^{(k)}_{n-k+1}\right).$$  
The final edge from $\av^{(k)}$ to $\av$ in the cycle implies $\st\left(a^{(k)}_2\cdots a^{(k)}_{n}\right)=\st(a_1\cdots a_{n-1})$.
Focusing on the standardizations of length $n-k$, we combine the previous two equalities to obtain the desired result:
$$\st(a_{k+1}\cdots a_n)=\st\left(a^{(k)}_2\cdots a^{(k)}_{n-k+1}\right)=\st(a_1\cdots a_{n-k}).$$
\end{proof}

\begin{corollary}\label{secondVertex}
For a closed $k$-walk $\left(\av^{(1)}, \av^{(2)}, \ldots, \av^{(k)}\right)$ with $\av^{(i)}=a^{(i)}_1a^{(i)}_2\cdots a^{(i)}_n$, the vertex $a^{(2)}$ satisfies the following equality: $\st\left(a^{(2)}_k\cdots a^{(2)}_n\right)=\st\left(a^{(1)}_1\cdots a^{(1)}_{n-k+1}\right)$.
\end{corollary}  

\begin{proof}
Following the same logic as the proof of Theorem~\ref{cycleNecessary}, we begin with the left-hand side and use subsequent edges to obtain the chain of equalities 
$$\st\left(a^{(2)}_k\cdots a^{(2)}_n\right)=\st\left(a^{(3)}_{k-1}\cdots a^{(3)}_{n-1}\right)=\cdots =\st\left(a^{(k)}_{2}\cdots a^{(k)}_{n-k+2}\right)=\st\left(a^{(1)}_1\cdots a^{(1)}_{n-k+1}\right).$$
\end{proof}

Before we prove a sufficient condition for when a vertex is contained within a closed $k$-walk, we first provide notation for two standardizations of substrings of a vertex $\av=a_1\cdots a_n$.  
We define $$y_{\av}=y_1\cdots y_{n-1}=\st(a_2\cdots a_n)$$ and 
$$z_{\av}=z_1\cdots z_{n-k+1}=\st(a_1\cdots a_{n-k+1}).$$  If $\bv=b_1\cdots b_n$ is the second vertex in a closed $k$-walk, it clearly must satisfy $\st(b_1\cdots b_{n-1})=y_{\av}$ to have an edge from $\av$ to $\bv$, and it also must have $\st(b_k\cdots b_n)=z_{\av}$ according to Corollary~\ref{secondVertex}.  
To clarify this point, Figure~\ref{yaza} provides an example of vertices $\av$ and $\bv$ when the permutations have length $n=7$ and $k=4$ along with $y_\av$ and~$z_\av$ which are placed over or under their corresponding $a_i$  and $b_i$ values.
These standardizations will be used frequently in the upcoming results of this section and as such we will use the next result without citing it in many of our results.

\begin{figure}[htb]
\begin{center}
\includegraphics[width=2in]{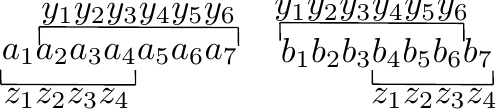}
\end{center}
\caption{Permutations $\av$ and $\bv$ with their corresponding $y_\av$ and $z_\av$}\label{yaza}
\end{figure}

\begin{theorem}\label{walkExistence}
Let $\av=a_1a_2\cdots a_n$ be a vertex in $G(n)$. If $\st(a_1a_2\cdots a_{n-k})=\st(a_{k+1}a_{k+2}\cdots a_n)$ then there exists a closed $k$-walk starting at $\av$.
\end{theorem}

\begin{proof}
We first find a vertex $\av^{(2)}=b_1\cdots b_n$ that is the second vertex on the closed $k$-walk starting at $\av$.  We must have $\st(b_1\cdots b_{n-1})=y_{\av}=\st(a_2\cdots a_n)$ and $\st(b_k\cdots b_n)=z_{\av}=\st(a_1\cdots a_{n-k+1})$.  To determine $\av^{(2)}$, we pursue one of the following procedures depending on the last element of $z_{\av}$:
\begin{itemize}
\item[$(A)$] If $z_{n-k+1}=1$, then set $b_n=1$ and $b_i=y_{i}+1$ for $i=1,\ldots, n-1$.

\item[$(B)$] Otherwise, let $\ell$ be the index with $z_{\ell}=z_{n-k+1}-1$.  Set $b_i=y_{i}$ for all indices $i$ with $y_{i}\leq y_{\ell+k-1}$, $b_n=y_{\ell+k-1}+1$, and if $z_{n-k+1}<n-k+1$, set
$b_j=y_{j}+1$ for all remaining elements.
\end{itemize}

In Case $(A)$,
$\st(b_1\cdots b_{n-1})=y_{\av}$ since each $b_i$ is simply attained by increasing $y_i$ by $1$.  We also have $\st(b_1\cdots b_{n-1})=y_{\av}$ in Case $(B)$ because the values within $\av^{(2)}$ and $y_{\av}$ are identical from $1$ to $y_{\ell+k-1}$ and are attained by increasing $y_i$ by $1$ for the remaining elements in $\av^{(2)}$.  Because of the assumption of 
$\st(a_1a_2\cdots a_{n-k})=\st(a_{k+1}a_{k+2}\cdots a_n)$, we have
$\st(y_{k}\cdots y_{n-1})=\st(z_1\cdots z_{n-k})$.  Therefore, $\st(b_k\cdots b_{n-1})=\st(z_{1}\cdots z_{n-k})$ since we already demonstrated that $\st(b_1\cdots b_{n-1})=y_{\av}$.  If Case $(A)$ was used, we can extend this standardization equality to be $\st(b_k\cdots b_n)=z_{\av}$ since both $b_n$ and $z_{n-k+1}$ are the smallest elements in these strings, respectively.  In Case $(B)$ with $z_{n-k+1}=n-k+1$, we obtain the same equality since $b_n$ and $z_{n-k+1}$ are the largest elements in those strings.  Finally, in Case $(B)$ with $z_{n-k+1}<n-k+1$, if $i_1$ and $i_2$ are the indices with $z_{i_1}+1=z_{n-k+1}=z_{i_2}-1$, then $b_n$ is assigned so that $b_{i_1}<b_n<b_{i_2}$.  These inequalities imply that the standardization equality can also be extended in this case to $\st(b_k\cdots b_n)=z_{\av}$.  Thus, each case maintains both of the required standardization equalities involving $y_{\av}$ and $z_{\av}$.

We now have that there is an edge from $\av$ to $\av^{(2)}$, and it remains to show that we can continue with this walk and return to $\av$ after a total of $k$ steps.  The procedure outlined above, using $y_{\av^{(2)}}$ and $z'_{\av}=z_1z_2\cdots z_{n-k+1}z_{n-k+2}$ where $z_{n-k+1}$ is replaced with $z_{n-k+2}$ in the breakdown of Cases $(A)$ or $(B)$, can be applied again to find a vertex $\av^{(3)}$.  Let $\av^{(t)}=a_1^{(t)}a_2^{(t)}\cdots a_n^{(t)}$ for each $t\in\{3,\ldots,k\}$. There is an edge connecting $\av^{(2)}$ to $\av^{(3)}$ in which $\st(a^{(3)}_1\cdots a^{(3)}_{n-1})=y_{\av^{(2)}}$ and $\st(a^{(3)}_{k-1}\cdots a^{(3)}_n)=\st(a_1\cdots a_{n-k+2})$.
 
 We continue this process of determining a vertex $\av^{(t)}$ such that $\st(a^{(t)}_1\cdots a^{(t)}_{n-1})=y_{\av^{(t-1)}}$ and $\st(a^{(t)}_{k-t+2}\cdots a^{(t)}_n)=\st(a_1\cdots a_{n-k+t-1})$.  When we reach $t=k$, we have that the vertex $\av^{(k)}$ satisfies $\st(a^{(k)}_{2}\cdots a^{(k)}_n)=\st(a_1\cdots a_{n-1})$.  This implies there is a directed edge from $\av^{(k)}$ to $\av$, completing the closed $k$-walk $(\av, \av^{(2)}, \av^{(3)}, \ldots, \av^{(k)})$. 
\end{proof}

\begin{corollary}\label{finishWalk}
For vertices $\av=a_1a_2\cdots a_n$ and $\bv=b_1b_2\cdots b_n$, if $\st(a_{t+1}\cdots a_n)=\st(b_1\cdots b_{n-t})$, then there is a $t$-walk from $\av$ to $\bv$.
\end{corollary}
\begin{proof}
The $t$-walk $(\av, \av^{(2)},\ldots,\av^{(t)},\bv)$ can be created by following a similar process as described in Theorem~\ref{walkExistence}.  The cases for finding the second vertex are determined by the last element of $\st(b_{1}\cdots b_{n-t})$ similar to when using Cases $(A)$ and $(B)$ in Theorem~\ref{walkExistence}.
\end{proof}

Note that by the symmetric property of permutations the following result holds.

\begin{proposition}\label{cycleComplement}
$(\av^{(1)},\av^{(2)},\ldots,\av^{(k)})$ is a closed $k$-walk in $G(n)$ if and only if $(\overline{\av^{(1)}},\overline{\av^{(2)}},\ldots,\overline{\av^{(k)}})$ is a closed $k$-walk.
\end{proposition}

Notice that Theorems~\ref{cycleNecessary} and~\ref{walkExistence} are converses of each other.  However, Theorem~\ref{cycleNecessary} (and Proposition~\ref{cycleComplement}) can be stated for $k$-cycles since each $k$-cycle is by definition a closed $k$-walk, but Theorem~\ref{walkExistence} cannot be extended to guarantee the existence of a $k$-cycle.  Part of the next result is an example of why such a cycle might not exist.

\begin{theorem}\label{trivialVertices}
Let $\av=1\,2\cdots n$ and $\bv=n\,(n-1)\cdots 1$ be vertices in $G(n)$.  Then
\begin{enumerate}[(a)]
\item the vertices $\av$ and $\bv$ are the only vertices contained in a $1$-cycle, and 

\item the vertices $\av$ and $\bv$ are not contained within a $k$-cycle for $2\leq k<n$.
\end{enumerate}
\end{theorem}
\begin{proof}
For part $(a)$, Theorem~\ref{cycleNecessary} implies when $k=1$ that the first $n-1$ and last $n-1$ elements of a vertex in a $1$-cycle must have the same standardization.  This property is only true for the two trivial vertices $\av$ and $\bv$.

For part $(b)$, by Proposition~\ref{cycleComplement}, we only need to consider $\av=1\,2\cdots n$ since $\bv=\overline{\av}$.  The vertex $\av$ clearly satisfies the assumption of Theorem~\ref{walkExistence}; hence, there is a closed $k$-walk starting at $\av$.  If we assume $\cv=c_1\cdots c_n$ is the second vertex of this walk, then since $y_{\av}=1\cdots (n-1)$ and $z_{\av}=1\cdots (n-k+1)$, we get $\st(c_1\cdots c_{n-1})=1\cdots (n-1)$ and $\st(c_k\cdots c_n)=1\cdots (n-k+1)$.  Since $k<n$, these two standardizations overlap and imply that $\cv=1\,2\cdots n$.  Thus, the closed $k$-walk is simply a $1$-cycle repeated $k$ times instead of a $k$-cycle.
\end{proof}

The trivial permutations are not the only vertices which satisfy the sufficiency condition for a closed $k$-walk and are not included in a $k$-cycle, as shown in the following example.

\begin{example}{\rm
Consider the vertex $\av=162534$.  For $k=4$, $\av$ satisfies the condition in Theorem~{\normalfont\ref{walkExistence}} of $\st(16)=\st(34)$.  This vertex is not included within a $4$-cycle though, as the only closed $4$-walk that contains it is the repetition of the $2$-cycle $(162534, 615243)$.}
\end{example}

To attempt to count how many $k$-cycles are in $G(n)$, we first need to examine how many closed 
walks or cycles can be created through a given vertex.  The following results will provide conditions 
for when a vertex is adjacent to multiple vertices whose directed edges begin closed $k$-walks.

\begin{theorem}\label{multipleVertices}
Let $k<n<2k$ and $\av=a_1\cdots a_n$ be a vertex in $G(n)$ contained within at least one closed $k$-walk.  Let $2\leq m\leq k$ be an integer. If $z_{n-k+1}=1$ and the elements in $\{1,\ldots, m-1\}$ are within positions $1,\ldots ,k-1$ of $y_{\av}$, or likewise if $z_{n-k+1}=n-k+1$ and $\{n-m+2,\ldots, n\}$ are within positions $1,\ldots ,k-1$ of $y_{\av}$, then the vertex $\av$ is contained in $m$ closed $k$-walks with distinct vertices adjacent to $\av$.
\end{theorem}

\begin{proof}
We will prove the result for the case of $y_{1},\ldots, y_{k-1}$ including $1,\ldots, m-1$, and the other case follows by symmetry.  Assume $\av$ is a vertex in a closed $k$-walk $\left(\av, \av^{(2)}, \av^{(3)}, \ldots, \av^{(k)}\right)$, and let $\av^{(2)}=b_1b_2\cdots b_n$. By Corollary~\ref{secondVertex}, we have $\st(b_1\cdots b_{n-1})=y_{\av}=y_1\cdots y_{n-1}$ and $\st(b_k\cdots b_n)=z_{\av}=z_1\cdots z_{n-k+1}$.  Let $i_1,\ldots, i_{m-1}$ be the indices containing the smallest $m-1$ elements of $y_1\cdots y_{n-1}$.  Then the positions $b_{i_1},\ldots, b_{i_{m-1}}$ are the smallest elements in $b_1\cdots b_{k-1}$, and 
$b_n$ is the smallest element in $b_{k}\cdots b_n$ since $z_1\cdots z_{n-k+1}=\st(b_k\cdots b_n)$.  
Since these portions of $\av^{(2)}$ do not overlap, we can choose $b_n$ to be any element in $\{1,\ldots, m\}$ and place the remaining elements from $\{1,\ldots,m\}$ in the positions $b_{i_1},\ldots, b_{i_{m-1}}$ according to the standardization $y_1\cdots y_{n-1}$.  The other elements larger than~$m$ will be fixed based on the standardization $y_1\cdots y_{n-1}$.  With $m$ different options for the vertex $\av^{(2)}$, we obtain $m$ distinct closed $k$-walks that include the vertex $\av$ by following the process in the proof of Theorem~\ref{walkExistence} to complete the closed walk.
\end{proof}

The assumption of $k<n<2k$ in the previous theorem is necessary for the existence of the conditions on the standardizations of $\av=a_1a_2\cdots a_n$.  In the case of $n\geq 2k$, $z_{n-k+1}=1$ implies that $a_{n-k+1}<a_i$ for all $i<n-k+1$.  If $y_j=1$ for some $j\in \{1,\ldots k-1\}$, then $a_j<a_{n-k+1}$ since $n-k+1\geq 2k-k+1=k+1$, which is a contradiction.

Theorem~\ref{multipleVertices} only guarantees the existence of $m$ closed $k$-walks rather than $k$-cycles, because it is possible that some of these $k$-walks are combinations of shorter cycles, as shown in the following example.

\begin{example}\label{multipleWalks}{\rm
For $n=5$ and $k=4$, the vertex $21435$ satisfies the assumptions of Theorem~\ref{multipleVertices} since $z_2=1$ and $y_\av=1324$, resulting in $m=4$ distinct closed $4$-walks that include this vertex.  However, one of these is a repetition of the $2$-cycle $(21435,13254)$ and only $3$ are $4$-cycles: $(21435, 14253, 31425, 13254)$, $(21435, 14352, 32415, 23154)$, and $(21435, 24351, 32415, 23154)$.}
\end{example}

The vertices satisfying the conditions in Theorem~\ref{multipleVertices} are not the only vertices that are contained in multiple closed walks.  The following result provides further conditions for this situation.  An example of a vertex that fulfills the criteria for this next theorem is included in Example~\ref{manyWalks}.

\begin{theorem}\label{otherMultipleVertices}
Let $\av=a_1 a_2\cdots a_n$ be a vertex in $G(n)$ that is contained within at least one closed $k$-walk, and let $2\leq m\leq k$. Assume there exists some indices $i, j, \ell_1,\ell_2,\ldots, \ell_{m-1}$ with $k\leq i,j\leq n-1$ and 
$1\leq \ell_1,\ell_2,\ldots, \ell_{m-1}\leq k-1$ for which 
$y_i= y_{\ell_1}-1=y_{\ell_2}-2=\ldots= y_{\ell_{m-1}}-(m-1)= y_j-m$ 
and $z_{i-k+1}+1= z_{n-k+1}= z_{j-k+1}-1$.  Then $\av$ is contained in $m$ closed $k$-walks with distinct vertices adjacent to $\av$.
\end{theorem}

\begin{proof}
Let $(\av,\av^{(2)},\av^{(3)},\ldots,\av^{(k)})$ be a closed $k$-walk containing $\av$, and let 
\[
\av^{(2)}=b_1b_2\cdots b_n.
\]  
Then by Corollary~\ref{secondVertex} we have $y_\av=\st(b_1b_2\cdots b_{n-1})$ and $z_{\av}=\st(b_{k}b_{k+1}\cdots b_n)$. 
Following the proof of Theorem~\ref{walkExistence}, since $z_{n-k+1}\neq 1$, we would use Case $(B)$ with $\ell=i-k+1$.  
This results in a vertex $\av^{(2)}$ with $b_s=y_{s}$ for all $s\in\{1,\ldots, n-1\}\setminus \{\ell_1,\ldots,\ell_{m-1},j\}$, $b_n=y_i+1$, and $b_s=y_{s}+1$ for $s\in \{\ell_1,\ldots,\ell_{m-1},j\}$.

Based on our assumptions for $y_{\av}$ and $z_{\av}$, we have 
$b_{i}\leq b_{\ell_1}\leq\cdots \leq b_{\ell_{m-1}}\leq b_{j}$ and $b_{i}\leq b_n\leq b_{j}$. 
 However, there is no necessary inequality relating $b_n$ with $b_{\ell_1},\ldots, b_{\ell_{m-1}}$. 
 Thus, we can permute the sequence $(b_{\ell_1},\ldots, b_{\ell_{m-1}},b_n)$ to result 
 in $m-1$ different sequences of the form $(c_{\ell_1},\ldots, c_{\ell_{m-1}},c_n)$ that maintains the inequalities 
 $c_{\ell_1}\leq\cdots \leq c_{\ell_{m-1}}$ with the element 
 $c_n=b_{\ell_r}$ for some $1\leq r\leq m-1$.  As such, there is a 
 directed edge from $\av$ to each of the distinct $m-1$ vertices of the form
 $\av^{(2)}=d_1d_2\cdots d_n$ where
$$d_s=\begin{cases}
b_s \hspace{.5cm} \text{if }s\notin \{\ell_1,\ldots, \ell_{m-1}, n\}\\
c_s \hspace{.5cm} \text{if }s\in  \{\ell_1,\ldots, \ell_{m-1}, n\}.\\
\end{cases}
$$
Since the equality $c_n=b_{\ell_r}$ satisfies $b_{i}\leq c_n\leq b_{j}$, swapping $b_n$ with the elements $b_{\ell_r}$ won't change the fact that $\st(d_{k}\cdots d_n)=\st(a_1\cdots a_{n-k+1})$.  Thus, by Corollary~\ref{finishWalk} we can find the remaining $k-2$ vertices to complete a closed $k$-walk from each of those $m-1$ vertices, giving us a total of $m$ closed $k$-walks with distinct second vertices.
\end{proof}

As shown in the next theorem, the two conditions in Theorems~\ref{multipleVertices} and~\ref{otherMultipleVertices} are the only way for $\av$ to be contained in multiple closed $k$-walks with differing vertices adjacent to $\av$.

\begin{theorem}\label{reverse}
Let $2\leq m\leq k$. Let $\av=a_1 a_2\cdots a_n$ be a vertex in $G(n)$ that is contained within $m$ closed $k$-walks with distinct elements adjacent to $\av$. Then one of the following is true regarding the standardizations $y_{\av}$ and $z_{\av}$:
\begin{enumerate}[(a)]
\item $z_{n-k+1}=1$ and $\{1,\ldots, m-1\}$ are within positions $1,\ldots ,k-1$ of $y_{\av}$, or $z_{n-k+1}=n-k+1$ and $\{n-m+2,\ldots, n\}$ are within positions $1,\ldots ,k-1$ of $y_{\av}$,

\item There exists some $i, j, \ell_1,\ell_2,\ldots, \ell_{m-1}$ with $k-1<i,j\leq n-1$ and 
$1\leq \ell_1,\ell_2,\ldots, \ell_{m-1}\leq k-1$ for which 
$y_i= y_{\ell_1}-1=y_{\ell_2}-2=\cdots= y_{\ell_{m-2}}-(m-1)= y_j-m$ 
and $z_{i-k+1}+1= z_{n-k+1}= z_{j-k+1}-1$.
\end{enumerate}
\end{theorem}

\begin{proof}
Suppose that $\av$ is contained in $m$ closed $k$-walks, say $\left(\av,\av^{(2,r)},\av^{(3,r)},\ldots,\av^{(k,r)}\right)$ for $r\in\{1,\ldots, m\}$.  Let $\av^{(2,r)}=a_1^{(2,r)}\;a_2^{(2,r)}\;\cdots\; a_n^{(2,r)}$.  
Since each $\av^{(2,r)}$ is in a closed $k$-walk with $\av$, we have the following for each $r$:
 \begin{align}
 \label{*}
  \st\left(a_1^{(2,r)}\cdots a_{n-1}^{(2,r)}\right)&=y_{\av}\\
  \label{**}
 \st\left(a_k^{(2,r)}\cdots a_n^{(2,r)}\right)&=z_{\av}.
 \end{align}
Equation~\eqref{*} implies that the vertex $\av^{(2,r)}$ is determined once a choice is made for the final element of $a_{n}^{(2,r)}$, and it is assumed that there are $m$ such choices.  Together, the two equalities cause the elements $a_{k}^{(2,r)},\ldots, a_{n-1}^{(2,r)}$ to be identical for all $m$ vertices; hence, the only differences between the $m$ vertices are possibly in the final position and in the first $k-1$ positions.

Consider one of the vertices $\av^{(2,1)}$, which we will simplify notationally as $\av^{(2)}$.  If $z_{n-k+1}=1$, then by Equation~\eqref{**}, $a_{n}^{(2)}<a_{s}^{(2)}$ for all $s\in\{k,\ldots, n-1\}$.  Therefore $a_{n}^{(2)}$ can only be interchanged with any of the first~$k-1$ elements if those elements are also smaller than each $a_{s}^{(2)}$ for all $s\in\{k,\ldots, n-1\}$.  The only way for this to produce $m-1$ other vertices that maintains the Equation~\eqref{*} is for there to be the elements $1,\ldots, m-1$ in the first $k-1$ positions of $y_{\av}$.  Then $a_{n}^{(2)}$ can be swapped with any of those $m-1$ positions, followed by a shuffling of the $m-1$ elements to keep their proper order.  This is exactly the scenario described in Condition~$(a)$ of this theorem.  If $z_{n-k+1}=n-k+1$, it follows analogously by symmetry.

If $z_{n-k+1}\neq 1,n-k+1$, then consider the positions in $z_{\av}$ that are 1 greater and 1 less than $z_{n-k+1}$, whose indices we will denote as $i-k+1$ and $j-k+1$ respectively where $k-1<i,j\leq n-1$.  These indices lie within the overlap of the two Equations~\eqref{*} and \eqref{**} 
and correspond to the positions $i$ and $j$ in $y_{\av}$, meaning $z_{i-k+1}$ and~$y_i$ match up with $a_{i}^{(2)}$ and likewise for $j$ within the standardizations.  Since $a_{n}^{(2)}$ must remain between $a_{i}^{(2)}$ and~$a_{j}^{(2)}$, then it can only be interchanged with elements in the first $k-1$ positions.  In order to keep Equation~\eqref{*} true and have $m-1$ other vertices that are part of a closed $k$-walk, then if these elements are in positions $\ell_1,\ldots, \ell_{m-1}$, we must have $1\leq \ell_s\leq k-1$ and $y_i<y_{\ell_s}<y_j$ for each $s\in\{1,\ldots,m-1\}$.  As before, we can create $m-1$ different vertices by having $a_{n}^{(2)}$ swapped with any of these $m-1$ elements and then permuting those $m-1$ elements to maintain their required order.  This set up is described in Condition $(b)$ of this theorem.  With every possibility for $z_{n-k+1}$ covered, the result is proven.
\end{proof}

We have now fully characterized how a vertex $\av$ can branch off to $m$ different vertices to begin closed $k$-walks, and in many cases, $k$-cycles.  This does not mean there are only $m$ sets of vertices containing $\av$ that can be connected to form a closed $k$-walk, since further branching can occur from vertices later in the walk, as illustrated by the following example.

\begin{example}\label{manyWalks}{\rm
Consider the vertex $\av=14263758$ with the cycle length of $k=6$.  This vertex satisfies the condition in Theorem~\ref{otherMultipleVertices} with $m=3$.  Since $z_{\av}=132$ and $y_{\av}=3152647$, we have $i=6$, $j=7$, $\ell_1=3$, and $\ell_2=5$, leading to three closed $6$-walks containing distinct vertices adjacent to $\av$.  This vertex, however, is actually contained in five closed $6$-walks using those three options for the second vertex:
\begin{align*}
\text{the repetition of the $2$-cycle } &(14263758, 31527486),\\
\text{this $2$-cycle followed by the $4$-cycle } &(14263758, 31627485, 15263748, 41527386), \text{ and}\\
\text{the three distinct $6$-cycles: } &(14263758, 31526487, 14253768, 31427586, 13264758, 21537486),\\
&(14263758, 31627485, 15263748, 41627385, 15263748, 41527386),\\
&(14263758, 31627485, 15263748, 51627384, 15263748, 41527386).
\end{align*}
}
\end{example}

Thus far, our results are centered on when vertices are within closed walks and cycles, and on how many vertices that a vertex can branch off to form other closed walks.  However, we also need to consider when edges exist with the same head and tail between consecutive vertices in a cycle.  The following result presents a condition for when adjacent vertices are joined by more than one edge.

\begin{lemma}\label{doubleEdges}
In $G(n)$, there are two directed edges from vertex $\av$ to vertex $\bv$ if and only if $\av=a_1a_2\cdots a_n$ and $\bv=\sigma(\av)=a_2a_3\cdots a_na_1$.

\end{lemma}

\begin{proof}
If we assume $\av=a_1a_2\cdots a_n$ and $\bv=a_2a_3\cdots a_na_1$, then it can easily be verified that the following are  edges from $\av$ to $\bv$: $a_1c_2c_3\cdots c_n(a_1+1)$ and $(a_1+1)c_2c_3\cdots c_na_1$ where 
$$ c_i=\begin{cases}
a_i+1 &\mbox{if $a_i>a_1$}\\
a_i &\mbox{if $a_i<a_1$.}\\
\end{cases}$$

Now we assume that there are two directed edges from the vertex $\av$ to $\bv$, labeled as 
$\wv=w_1w_2\cdots w_nw_{n+1}$ and $\xv=x_1x_2\cdots x_nx_{n+1}$.  
Then $\st(w_1\cdots w_n)=\st(x_1\cdots x_n)=\av$ and 
$\st(w_2\cdots w_{n+1})=\st(x_2\cdots x_{n+1})=\bv$.  
We claim that $w_2\cdots w_n=x_2\cdots x_n$.  If there is an $i\in$~$\{2,\ldots,n\}$ such that 
$w_i\neq x_i$, then without loss of generality, we can assume $w_i<x_i$.  Then there must 
also be some $j\in \{1,\ldots, n+1\}$ with $w_j>x_j$ for which $w_i<w_j$ and $x_j\leq w_i<x_i$.  
If $j\in\{1,\ldots, n\}$, then the opposite direction of the inequalities relating the $i$th and $j$th elements of $\wv$ and $\xv$ cause a change in the ordering of the first $n$ elements, that is, 
$\st(w_1w_2\cdots w_n)\neq \st(x_1x_2\cdots x_n)$.  Otherwise, if $j=n+1$, then this change in order results in $\st(x_2\cdots x_nx_{n+1})\neq \st(w_2\cdots w_nw_{n+1})$.  
In either case, this contradicts the fact that both edges are from $\av$ to $\bv$, and thus we have 
$w_2\cdots w_n=x_2\cdots x_n$.  With this equality, the only way for the two edges to be 
distinct is for $x_1=w_{n+1}$ and $x_{n+1}=w_1$. This implies that the vertices the two 
edges connect are equal to the standardizations 
$a_1\cdots a_n=\st(w_1\cdots w_n)$ and $b_1\cdots b_n=\st(w_2\cdots w_nw_1)$, where 
the edge $\wv$ is used for the first equality and $\xv$ is used for the second one.  Therefore, 
the two vertices are of the expected form $\av=a_1a_2\cdots a_n$ and 
$\bv=a_2a_3\cdots a_na_1$.
\end{proof}

The final steps in the proof of Lemma~\ref{doubleEdges} describe how multiple edges connecting 
a pair of vertices in the same direction are identical with the exception of the first and last 
element being swapped.  This reasoning directly proves the following corollary.

\begin{corollary}\label{afterDoubleEdges}
There cannot be more than two edges in $G(n)$ from a vertex $\av$ to a vertex $\bv$.
\end{corollary}

We will see the conditions in which there is a directed edge from $\av$ to $\bv$ and a directed edge from $\bv$ to $\av$ in the next section (since these are $2$-cycles). 

We end this section with conditions in which permutations cannot possibly be in some closed $k$-walk, or likewise $k$-cycle, in $G(n)$. 

\begin{lemma}\label{necessaryPerm}
Let $\av=a_1a_2\cdots a_n$ be a vertex in $G(n)$. Let one of the following conditions be satisfied:
\begin{packedItem}
 \item[$(a)$] For some positive integer $t\geq 2$, $n\geq tk+1$, $n-(t-2)\in\{a_{k+1},a_{k+2},\ldots,a_{n-(t-1)k}\}$, and for each $j\in \{n-t+3,n-t+4,\ldots,n\}$, $j\not\in \{a_1,\ldots,a_k\}$.
 \item[$(b)$] For some positive integer $t\geq 2$, $n\geq tk+1$, $t-1\in\{a_{k+1},a_{k+2},\ldots,a_{n-(t-1)k}\}$, and for each $j\in \{1,2,\ldots,t-2\}$, $j\not\in \{a_1,\ldots,a_k\}$.
 \item[$(c)$] For some positive integer $t\geq 2$, $n\geq tk+1$, $n-(t-1)\in \{a_1,a_2,\ldots,a_s\}$ where $s=\min(\{k,n-tk\})$, and for each $j\in \{n-t+2,n-t+3,\ldots,n\}$, $j\not\in \{a_1,\ldots,a_k\}$.
 \item[$(d)$] For some positive integer $t\geq 2$, $n\geq tk+1$, $t\in \{a_1,a_2,\ldots,a_s\}$ where $s=\min(\{k,n-tk\})$, and for each $j\in \{1,2,\ldots,t-1\}$, $j\not\in \{a_1,\ldots,a_k\}$.
\end{packedItem}
Then $\av$ cannot be a vertex in any closed $k$-walks in $G(n)$.
\end{lemma}

\begin{proof}
We will proceed by contradiction, so suppose $\av$ is in some closed $k$-walk in $G(n)$. First suppose that Condition $(a)$ is satisfied. If $t=2$, that is, if $n\in \{a_{k+1},\ldots,a_{n-k}\}$, by Theorem~\ref{cycleNecessary} it is clear that an element larger than $n$ must appear as an element later in the permutation than where $n$ appeared in the permutation $\av$, but this is impossible, hence $n\in \{a_{n-k+1},a_{n-k+2},\ldots,a_n\}$.
We will proceed by induction on $t$, so assume that for each positive integer $\ell\in\{2,\ldots,t\}$, $\av$ cannot be a vertex in any closed $k$-walks in $G(n)$ when $n\geq \ell k+1$, $n-(\ell-2)\in\{a_{k+1},a_{k+2}\ldots,a_{n-(\ell-1))k}\}$ and for each $j\in\{n-\ell+3,\ldots,n\}$, $j\not\in\{a_1,\ldots,a_k\}$. Let $a_m\geq n-(t-2)$ where $k+1< m <n-(t-1)k$. If $a_{m+k}> a_m$ then $a_m$ is less than $a_{m+2k},a_{m+3k},\ldots,a_{m+(t-1)k}$, but there are only $t-2$ elements in $\{1,2,\ldots,n\}$ larger than $a_m$, so $a_{m+k}<a_m$. 
Since all elements in $\{a_1,\ldots,a_{m-1}\}$ are less than $a_m$, we have $a_{m-k}<a_m$, so by Theorem~\ref{cycleNecessary}, $a_m<a_{m+k}$, which is a contradiction. 
A similar argument can be made for Condition $(b)$ by Proposition~\ref{cycleComplement}.

Second suppose that Condition $(c)$ is satisfied, so let $a_m=n-(t-1)$. Then $a_{n-(t-1)+1}=a_{n-(t-2)}$ which is exactly the same as Condition $(a)$, so we reach a contradiction. A similar argument can be made for Condition~$(d)$ by Proposition~\ref{cycleComplement}.
\end{proof}

\section{Number of $2$-cycles}\label{2cycles}

We now focus our attention on $2$-cycles, which can be completely enumerated using the results from the previous section.  One important fact that we rely on in this section is Theorem~\ref{trivialVertices}$(a)$, which states that only the trivial vertices are contained in a $1$-cycle.  Therefore, any non-trivial vertex satisfying Theorem~\ref{walkExistence}, which guarantees the vertex is in closed $2$-walk, is in fact in a $2$-cycle.

Let $v_{n,k}$ denote the number of vertices contained in a $k$-cycle in $G(n)$, and let $C_{n,k}$ be the number of $k$-cycles.  We now present our results for the number of vertices in $2$-cycles, followed by the number of these $2$-cycles in the graph of overlapping permutations.  

\begin{theorem}\label{numberOf2cycleVertices}
For $n\geq 4$, the number of vertices contained in $2$-cycles in $G(n)$ is $v_{n,2}=2n+2$.
\end{theorem}

\begin{proof}
Let $\av=a_1a_2\cdots a_n$ be a vertex in $G(n)$ contained within a $2$-cycle.  By Theorem~\ref{cycleNecessary}, we know that $\st(a_1\cdots a_{n-2})=\st(a_3\cdots a_n)$. Looking at the first two elements of these strings, we see $\st(a_1 a_2)=\st(a_3 a_4)$, which results in only $12$ permutations of length $4$ that could possibly be $\st(a_1\cdots a_4)$: $1234$, $1324$ $1423$, $2143$, $2314$, $2413$, $3142$, $3241$, $3412$, $4132$, $4231$, $4321$. 

We will prove the result separately depending on whether $n$ is even or odd.  
 First suppose $n$ is even. Additionally, we first assume $\st(a_1\cdots a_4)=1423$. Since $\st(a_1\cdots a_{n-2})=\st(a_{3}\cdots a_n)$, it follows that 
\begin{align}\label{1423}
\st(a_1\cdots a_4)=\st(a_3\cdots a_6)=\st(a_5\cdots a_8)=\cdots=\st(a_{n-3}\cdots a_n)
\end{align}
are all equal to $1423$.
As a consequence, all even elements in $\av$ are larger than the odd elements, $a_2>a_4>a_6>\cdots> a_n$, and $a_1<a_3<a_5<\cdots< a_{n-1}$. Therefore, the only possible permutation that fits this criteria is $\av=1\; n\; 2\; (n-1)\; 3\cdots \frac{n}{2}\; (\frac{n}{2}+1)$. A similar argument can be made for vertices where $\st(a_1\cdots a_4)\in\{2143,2314,3241,3412,4132\}$ so that each of these standardizations correspond to exactly one distinct vertex in $G(n)$ that is contained in a $2$-cycle. 

Now suppose $\st(a_1\cdots a_4)=1234$. The equalities $\st(a_3\cdots a_6)=\cdots=\st(a_{n-3}\cdots a_n)=1234$ imply $a_1<a_2<a_3<\cdots<a_n$.  Thus, $1\,2\,3\cdots n$ is the only permutation which matches this standardization, but this is not a $2$-cycle by Theorem~\ref{trivialVertices}$(b)$. Similarly, the only vertex with $\st(a_1\cdots a_4)=4321$ is $\av=n\,(n-1)\,(n-2)\cdots 1$, which similarly does not generate a $2$-cycle. 

Suppose $\st(a_1\cdots a_4)=1324$. As reasoned above, $\st(a_3\cdots a_6)=\cdots=\st(a_{n-3}\cdots a_n)=1324$, so $a_1<a_3<\ldots< a_{n-1}$ and $a_2<a_4<\cdots< a_n$, and consequently, $a_1=1$, $a_3=2$, $a_{n-2}=n-1$, and $a_n=n$. If $a_2=\frac{n}{2}+2$ then $\{a_4,a_6,\ldots,a_n\}\subseteq\{\frac{n}{2}+3,\frac{n}{2}+4,\ldots,n\}$, but there are $\frac{n}{2}-1$ even positioned elements between $a_4$ and $a_n$ inclusively while there are $\frac{n}{2}-2$ values for the even positioned elements to choose from, thus $a_2\leq \frac{n}{2}+1$. In fact, for each $a_2\in\{3,\ldots,\frac{n}{2}+1\}$ there is exactly one permutation of length $n$ satisfying the inequalities $a_1<a_3<\cdots <a_{n-1}$ and $a_2<a_4<\cdots<a_n$, and Equation~\eqref{1423}. 
Thus, there are $\frac{n}{2}-1$ permutations with $\st(a_1\cdots a_4)=1324$. A similar argument can be made for permutations with $\st(a_1\cdots a_4)\in \{2413,3142,4231\}$. 
Therefore, when $n$ is even there are $4\left(\frac{n}{2}-1\right)+6=2n+2$ permutations that satisfy the conditions in Theorem~{\normalfont\ref{cycleNecessary}}.

Next suppose that $n$ is odd.  Theorem~\ref{cycleNecessary} implies the standardization equalities $$\st(a_1\cdots a_4)=\st(a_3\cdots a_6)=\cdots=\st(a_{n-4}\cdots a_{n-1})$$ $$\text{and }\st(a_2\cdots a_5)=\st(a_4\cdots a_7)=\cdots=\st(a_{n-3}\cdots a_{n}).$$
If $\st(a_1\cdots a_4)\in\{1423,2143,2314,3241,3412,4132\}$, then a similar argument as that made in the even $n$ case will show that there is exactly one permutation of length $n$ corresponding to each permutation of length~$4$.
 The same can be done for when $\st(a_1\cdots a_4)\in\{1234,4321\}$, although the trivial vertices associated with these standardizations are not within a $2$-cycle. 

Suppose $\st(a_1\cdots a_4)=1324$. As reasoned above, $\st(a_3\cdots a_6)=\cdots=\st(a_{n-4}\cdots a_{n-1})=1324$, so $a_1<a_3<\ldots< a_{n-2}$ and $a_2<a_4<\cdots< a_{n-1}$, and consequently, $a_1=1$ and $a_3=2$. If $a_2=\frac{n-1}{2}+3$ then 
$a_4,a_6,\ldots,a_n$ are values from the set $\{\frac{n-1}{2}+4,\frac{n-1}{2}+5,\ldots,n\}$, but there are more even positions in $\av$ not including $a_2$ than values to choose from, which is a contradiction. Thus $a_2\leq \frac{n-1}{2}+2$. For each $a_2\in\{3,4,\ldots,\frac{n-1}{2}+2\}$ there is exactly one permutation of length $n$ with the necessary characteristics. Therefore, there are $\frac{n-1}{2}$ permutations with $\st(a_1\cdots a_4)=1324$. We obtain the same number of permutations with $\st(a_1\cdots a_4)=4231$ by symmetry.

Suppose $\st(a_1\cdots a_4)=2413$. It is clear that $a_2=n$ and $a_4=n-1$. By using a similar argument to the one above, $a_1\geq \frac{n+1}{2}$ and $a_1\leq n-2$.
There is exactly one permutation for each $a_1\in \{\frac{n+1}{2},\frac{n+1}{2}+1,\ldots,n-2\}$ by the inequalities $a_1<a_3<\cdots <a_{n}$ and $a_2<a_4<\cdots<a_{n-1}$, and Equation~\eqref{1423}, so there are $n-2-\frac{n+1}{2}+1=\frac{n-1}{2}-1$ permutations with $\st(a_1\cdots a_4)=2413$.  By symmetry, there are also $\frac{n-1}{2}-1$ permutations with $\st(a_1\cdots a_4)=3142$.  In total, when $n$ is odd there are $2\left(\frac{n-1}{2}\right)+2\left(\frac{n-1}{2}-1\right)+6=2n+2$ vertices contained in a $2$-cycle, completing our second and final case.
\end{proof}

An interesting property to note regarding vertices $\av=a_1\cdots a_n$ within $2$-cycles is that they are alternating permutations; that is,  if $a_i<a_{i+1}$ when $i$ is even and $a_i>a_{i+1}$ when $i$ is odd, or if $a_i>a_{i+1}$ when $i$ is even and $a_i<a_{i+1}$ when $i$ is odd.  It can also be stated as $\av$ having ascents only in even positions and descents only in odd positions, or vice versa.

Now that we know how many vertices are contained in $2$-cycles, we will prove how many $2$-cycles exist in~$G(n)$.  This time the cases of even and odd length permutations result in different enumerations.

\begin{theorem}\label{numberOf2cycles}
For $n\geq 4$, the number of $2$-cycles in $G(n)$ is the following
$$C_{n,2}=\begin{cases}
n+2 &\mbox{if $n$ is even}  \\
n+3 &\mbox{if $n$ is odd.}    \\
\end{cases}$$

Furthermore, when $n$ is even, there are no $2$-cycles featuring a multiedge with the same head and tail, and each vertex is contained in one $2$-cycle except for two vertices that are contained in exactly two $2$-cycles.  These vertices are $\av=1\;(\frac{n}{2}+1)\;2\;(\frac{n}{2}+2)\cdots (\frac{n}{2}-1)\;(n-1)\;\frac{n}{2}\;n$ and $\overline{\av}$.

When $n$ is odd, each vertex is paired with only one other vertex for form a $2$-cycle, and
exactly two cycles contain edges with the same head and tail, namely from $\av=\frac{n+1}{2}\; 1 \;\frac{n+3}{2}\, 2\;\frac{n+5}{2} \cdots (n-1)\;\frac{n-1}{2}\;n $ to $\sigma(\av)$ and from $\overline{\av}$ to $\sigma(\overline{\av})$.  
\end{theorem}

\begin{proof}

From the previous theorem, we know in each case that $C_{n,2}\geq n+1$ once these $2n+2$ vertices are paired up into $2$-cycles.

We first assume $n$ is even and begin by showing there does not exist any distinct cycles of the form $(\av,\bv)$ and $(\av,\bv)$ when $n$ is even; that is, there are no pairs of vertices connected by two edges in the same direction. By Lemma~{\normalfont\ref{doubleEdges}}, if such a pair exists, then $\av=a_1\cdots a_n$ and $\bv=a_2a_3\cdots a_n a_1$. Since $\st(a_1\cdots a_{n-2})=\st(a_3\cdots a_n)$, $a_1>a_3>a_5>\cdots>a_{n-1}$ or $a_1<a_3<a_5<\cdots<a_{n-1}$. Likewise, $\st(a_3\cdots a_n a_1)=\st(a_1a_2\cdots a_{n-1})$ implies that $a_3>a_5>\cdots > a_{n-1} > a_1$ or $a_3<a_5<\cdots<a_{n-1}<a_1$, which in either case is a contradiction.  Thus, there are no $2$-cycles with edges that have the same head and tail in $G(n)$ when $n$ is even.

We now consider vertices $\av$ that are contained in multiple $2$-cycles with distinct second vertices.  One can observe that $1\; (\frac{n}{2}+1)\; 2\; (\frac{n}{2}+2)\cdots(\frac{n}{2}-1)\;(n-1)\;\frac{n}{2}\; n$ and $n\;\frac{n}{2}\;(n-1)\;(\frac{n}{2}-1)\cdots(\frac{n}{2}+2)\;2\;(\frac{n}{2}+1)\;1$ are each in two such $2$-cycles according to Theorem~\ref{otherMultipleVertices} with $m=2$. Since these were included in the $2n+2$ vertices, the number of $2$-cycles in $G(n)$ is now at least $n+2$. 

It remains to show that there are no other vertices that are in two or more edge-disjoint $2$-cycles.  By Theorem~\ref{reverse}, if such a vertex exists, then it must satisfy the conditions of either Theorem~\ref{multipleVertices} or~\ref{otherMultipleVertices}.  The former of these theorems clearly does not apply when $n\geq 4$ and $k=2$ since it requires $n<2k$.  Then we assume there is a vertex $\av=a_1\cdots a_n$ satisfying the assumptions of Theorem~\ref{otherMultipleVertices} for $k=2$, which based on the cycle length would require $m=2$ and the index $\ell=1$.  To satisfy the condition, the values $y_1$ and $z_{n-1}$ cannot be $1$ or~$n-1$.  Recall from the proof of Theorem~\ref{numberOf2cycleVertices} that $\st(a_1\cdots a_4)=\st(a_3\cdots a_6)=\cdots \st(a_{n-3}\cdots a_n)$ and these are equal to one of ten possible length $4$ permutations.  If we first assume $\st(a_1\cdots a_4)=3142$, we obtain inequalities from these standardizations of the form $a_1<a_3<\cdots <a_{n-1}$ and $a_2<a_4<\cdots <a_n$.  Since $a_{n-1}>a_n$ because $\st(a_{n-3}\cdots a_n)=3142$, the inequalities imply $a_{n-1}=n$, hence $z_{n-1}=n-1$, which we stated could not occur when satisfying Theorem~\ref{otherMultipleVertices}.  Using the analogous inequalities for the following seven other options $\{1423, 2143, 2314, 2413, 3142, 3241, 4132\}$ for the standardization of the first four elements, we arrive at similar conclusions that $a_2$ or $a_{n-1}=1$ or $n$, causing $y_1$ or $z_{n-1}$ to be $1$ or $n-1$.  

After eliminating these possibilities, we now have that a vertex $\av$ contained within two distinct $2$-cycles must satisfy $\st(a_1\cdots a_4)=1324$ or $4231$, which we note are the cases for our two previously mentioned vertices that do satisfy Theorem~\ref{otherMultipleVertices}.  
We have from the assumptions of Theorem~\ref{otherMultipleVertices} that there is some $i,j$ between $2$ and $n-1$ with $y_i+1=y_1=y_j-1$ and $z_{i-1}+1=z_{n-1}=z_{j-1}-1$.  
Then since $\st(y_2\cdots y_{n-1})=\st(z_1\cdots z_{n-2})$ as a consequence of Theorem~\ref{cycleNecessary}, we claim that $y_2\cdots y_{n-1}=z_1\cdots z_{n-2}$, forcing $y_1=z_{n-1}$ as well.  If there was some differing element between the overlapping portion of $y_{\av}$ and $z_{\av}$, say $z_r>y_r$, the only way to maintain the standardization equality between these portions is for $z_{n-1}$ to correspondingly decrease, or increase if $z_r<y_r$.  Either situation would prevent $z_{i-1}+1=z_{n-1}=z_{j-1}-1$, confirming our claim.

As shown in the proof of Theorem~\ref{numberOf2cycleVertices}, if $\st(a_1\cdots a_4)=1324$, then $a_1=1$ and $a_n=n$, resulting in $z_r=a_r$ and $y_r=a_{r+1}-1$ for each index $r=1,\ldots, n-1$.  Thus, we need $a_{2}-1=a_{n-1}$ to have $y_1=z_{n-1}$.  Additionally from that proof, we know $3\leq a_2\leq \frac{n}{2}+1$ along with the inequalities $a_1<a_3<\cdots <a_{n-1}$ and $a_2<a_4<\cdots <a_n$.  Since there are $\frac{n}{2}-1$ elements $a_3, a_5,\ldots, a_{n-1}$ in increasing order, all of which are larger than $a_1=1$, it follows that if $a_2<\frac{n}{2}-1$, it is not possible for $a_{n-1}<a_2$.  Thus, $a_2=\frac{n}{2}-1$ is the only possible vertex $\st(a_1\cdots a_4)=1324$ that satisfies Theorem~\ref{otherMultipleVertices}.  The case of $4231$ follows by symmetry.

This completes the proof for the even case, as the $2n+2$ vertices contained in $2$-cycles, added with the two vertices that are contained in two distinct $2$-cycles, are paired together make the number of $2$-cycles be $\frac{2n+4}{2}=n+2$. 

Now we assume that $n$ is odd. We will show that there does not exist any cycles of the form $(\av,\bv)$ and $(\av,\cv)$; that is, there is no vertex contained in two $2$-cycles with distinct second vertices. As with the even case, Theorem~\ref{reverse} implies such a vertex would satisfy Theorem~\ref{multipleVertices}, which does not apply when $k=2$, or Theorem~\ref{otherMultipleVertices}.  Once again, there are only ten potential permutations of length $4$ for $\st(a_1\cdots a_4)$.  If $\st(a_1\cdots a_4)=1324$, then unlike the even case, we will show that there are no vertices that satisfy Theorem~\ref{otherMultipleVertices}.  We have $a_1<a_3<\cdots <a_n$ and $a_2<a_4<\cdots <a_{n-1}$, in addition to the fact that Equation~\eqref{1423} 
implies $\st(a_{n-2}a_{n-1}a_n)=132$ in the $n$ odd case.  Hence $a_{n-1}>a_n$, combining with the inequalities to result in $a_{n-1}=n$ and $z_{n-1}=n-1$.  This vertex cannot satisfy Theorem~\ref{otherMultipleVertices}, and similar arguments can be made for each of the nine remaining length $4$ standardizations.  Thus, there are no vertices contained in multiple $2$-cycles with distinct second vertices.

It is clear that there are two edges with the same head and tail from $\av=1 \;\frac{n+3}{2}\, 2\;\frac{n+5}{2} \cdots (n-1)\;\frac{n-1}{2}\;n \; \frac{n+1}{2}$ to $\bv=\frac{n+3}{2}\, 2\;\frac{n+5}{2} \cdots (n-1)\;\frac{n-1}{2}\;n \; \frac{n+1}{2}\; 1$ that creates two $2$-cycles between these vertices.  Likewise, there are two $2$-cycles between their complements $\overline{\av}$ and $\overline{\bv}$. It remains to show that there are no other vertices that have this property. If there did exist another such pair, then by Lemma~\ref{doubleEdges}, the pair is $\av=a_1\cdots a_n$ and $\bv=a_2\cdots a_n a_1$, and the edge from $\bv$ to $\av$ would imply that $\st(a_1\cdots a_{n-1})=\st(a_3\cdots a_{n}a_1)$.  Additionally, since $\st(a_1\cdots a_{n-2})=\st(a_3\cdots a_{n})$, $a_1>a_3>\cdots>a_n$ or $a_1<a_3<\cdots <a_n$.  As a result, $a_{n-1}>a_1$ or $a_{n-1}<a_1$ respectively. It follows that $a_2>a_4>\cdots>a_{n-1}$ or $a_2<a_4<\cdots<a_{n-1}$.
The only permutations that fit all of these criteria above are the ones indicated. 

With all $2n+2$ vertices being paired with exactly one vertex to form a $2$-cycle, and with only two of these containing edges with the same head and tail, the number of $2$-cycles in $G(n)$ for odd $n$ is $\frac{2n+2}{2}+2=n+3$.
\end{proof}

Note that each of the previous two theorems assumed that $n\geq 4$.  When $n=3$, the results are not true, as seen in the example in Section 2, since there are only $4<2n+2$ vertices contained in $2$-cycles.  Even though the six $2$-cycles in $G(3)$ does match with the expected $n+3$ from Theorem~\ref{numberOf2cycles}, there are cycles of the form $(\av,\bv)$ and $(\av,\cv)$ with different pairs of vertices, which is a case that doesn't occur for larger odd values of $n$.

If you try to extend the proof technique from $k=2$ in Theorems~\ref{numberOf2cycleVertices} and \ref{numberOf2cycles} to $k=3$, the problem becomes quite difficult.  Instead of examining cases based on the standardization of the first $4$ elements, it would depend on the first $6$ elements.  There are $120$ permutations of length $6$ that could possibly start a permutation where the standardization of the first $n-3$ elements is the same as the standardization of the last $n-3$ elements.

\section{Number of Vertices Within Cycles}\label{countVertices}

In this section, we will show several ways of attaining or bounding the number of closed $k$-walks and the number of $k$-cycles when $k$ is prime.  We first define the number $w_{n,k}$ to be the number of vertices of $G(n)$ contained in a closed $k$-walk.

\begin{theorem}\label{countingVertices2k}
If $n\leq 2k$ then the number of vertices contained in a closed $k$-walk is $\displaystyle w_{n,k}=\frac{n!}{(n-k)!}$. 
\end{theorem}

\begin{proof}
Let $\av=a_1a_2\cdots a_n$ be a permutation contained in a closed $k$-walk, so $\st(a_1\cdots a_{n-k})=\st(a_{k+1}\cdots a_n)$. Since the first $n-k$ and last $n-k$ elements do not overlap when $n\leq 2k$, choosing the first $n-k$ elements is independent of choosing the middle $n-2(n-k)=2k-n$ elements. There are $n!/k!$ ways to choose the first $n-k$ elements and $\frac{k!}{(n-k)!}$ ways to choose the middle $2k-n$ elements. Once the first $n-k$ elements and middle $2k-n$ elements of a permutation are chosen, the remaining elements are all determined by Theorem~\ref{cycleNecessary}. Thus, there are $\left(\frac{n!}{k!}\right)\left( \frac{k!}{(n-k)!}\right)=\frac{n!}{(n-k)!}$ ways to build a permutation $\av$. 
\end{proof}

\begin{theorem}\label{countingVertices2k3k}
Let $n>2k$ be odd, $k\geq 3$, and $\av=a_1a_2\cdots a_n$ be a vertex in $G(n)$. 
Then the number of vertices contained in a closed $k$-walk is bounded above as follows:
\begin{align*}
w_{n,k}&\leq \frac{(n-2)!}{(n-k)!}\left(\left(n+\frac{1}{2}\right)(n+1)(n-1)+k+\frac{n-5}{2}-2-(n-1)\left\lceil \frac{n-1}{4}\right\rceil\right).
\end{align*}
\end{theorem}

\begin{proof}
Let $\av=a_1a_2\cdots a_n$ be a permutation in a closed $k$-walk, so $\st(a_1\cdots a_{n-k})=\st(a_{k+1}\cdots a_n)$ by Theorem~\ref{cycleNecessary}. 

Note that at times our analysis is done by assuming $a_1\leq n/2$. This is because by Proposition~\ref{cycleComplement}, there are the same number of permutations such that $a_1<n/2$ as there are there are with $a_1>n/2$.
We will further break this case down into smaller cases based on the relationship between $a_1$ and $a_2$. 

Suppose that $a_1$ is odd. There are $(n+1)/2$ choices for $a_1$ and $(n-1)!/((n-1)-(k-1))!=(n-1)!/(n-k)!$ choices for $a_2,a_3,\ldots,a_k$. However, some of these will not fit the standardization condition in Theorem~\ref{cycleNecessary}. For this reason, $w_{n,k}$ is an inequality rather than an equality. 

Suppose $a_1$ is even. If $a_2<a_1$ then there are at most $(n-2)!/((n-2)-(k-2))!=(n-2)!/(n-k)!$ ways to choose $a_3,\ldots,a_k$, $a_1-1$ ways to choose $a_2$, and $(n-1)/2$ ways to choose $a_1$.

Next suppose that $a_2>a_1$ and $a_1=2$. If $1\not\in \{a_3,\ldots,a_k\}$ then $a_{k+1}=1$, since otherwise $a_i>a_1$ for all $i\in\{k+1,k+2,\ldots,n\}$, which would not contain the element $1$. However, $1$ cannot be in $\{a_{k+1},\ldots,a_{n-k}\}$ by Lemma~\ref{necessaryPerm}, so $1\in\{a_3,\ldots,a_k\}$.
There are $k-2$ choices for the position of this element. So there are $(k-2)((n-2)!)/((n-2)-(k-2))!=(k-2)((n-2)!)/(n-k)!$ permutations with $a_2>a_1$ and $a_1=2$. 

Finally suppose that $a_2>a_1$ and $a_1\neq 2$. 
So for each odd $a_1\neq 2$, the number of permutations of this specified form is $(n-2)!/((n-2)-(k-2))!=(n-2)!/(k-2)!$.
Notice that we have $2(n-5)/4=(n-5)/2$ or $2(n-3)/4-1=(n-3)/2-1=(n-5)/2$ values for $a_1$ in this case if $n\equiv 1\pmod{4}$ or $n\equiv 3\pmod{4}$, respectively. The sum of these enumerations for the vertices possibly satisfying Theorem~\ref{cycleNecessary} provides an upper bound for the actual number of vertices in closed $k$-walks. Thus, the number of vertices contained in a closed $k$-walk is
\begin{align*}
w_{n,k}&\leq \frac{(n+1)((n-1)!)}{2((n-k)!)} 
+ 2\left(\frac{n-1}{2}\right)\sum_{i=1}^{\lceil(n-1)/4\rceil}(2i-1)\left(\frac{(n-2)!}{(n-k)!}\right)\\
&\phantom{\leq} \; + (k-2)\left(\frac{(n-2)!}{(n-k)!}\right) 
+ \left(\frac{n-5}{2}\right) \frac{(n-2)!}{(n-k)!} \\
&= \frac{(n-2)!}{(n-k)!}\left(\left(n+\frac{1}{2}\right)(n+1)(n-1)+k+\frac{n-5}{2}-2-(n-1)\left\lceil \frac{n-1}{4}\right\rceil\right).
\end{align*}

\end{proof}

When $n$ is even, it is more difficult to calculate a bound for the number of closed $k$-walks. For example, besides using a case-by-case argument for each $n$, it is not clear to the authors why there are $14$ vertices that are in closed $4$-walks when $n=10$, $a_1=3$, and $a_2=2$ while there are $56$ vertices that are in closed $4$-walks when $n=10$, $a_1=3$, and $a_2=1$. There are many other instances similar to this that make it difficult to count the number of closed $k$-walks when $n$ is even.

\begin{example}
{\rm
There are instances when the counting method in Theorem~\ref{countingVertices2k3k} is not enough to get an exact bound. One such example is when $n=11$ and $k=3$. The permutation $3615827a49b$ satisfies Theorem~\ref{cycleNecessary} but there is no permutation that satisfies Theorem~\ref{cycleNecessary} that has $a_1=3$, $a_2=6$, and $a_3=4$. In fact, this permutation is the only permutation that starts with $36$ and satisfies Theorem~\ref{cycleNecessary}. It is not clear at this time how to ensure that we account for each possible contingency in the way provided in the proof of Theorem~\ref{countingVertices2k3k}}.
\end{example}


We now examine the relationship of the cycle lengths if a vertex is in cycles of two different lengths. This will help in converting the results in the previous two theorems about closed $k$-walks into counting vertices in $k$-cycles, which was denoted as $v_{n,k}$.  

\begin{theorem}\label{combiningCycles}
Assume $\av=a_1a_2\cdots a_n$ is a vertex that is in some $k$-cycle in $G(n)$.  If $\gcd(k,j)=1$ and $k+j<n$, then $\av$ is not in a $j$-cycle.
\end{theorem}
\begin{proof}
Assume that $\av$ is in both an $k$-cycle and a $j$-cycle, hence by Theorem~\ref{cycleNecessary}, $\av$ satisfies the following: 
\begin{align}
\st(a_1\cdots a_{n-k})&=\st(a_{k+1}\cdots a_{n}) \label{eq_i}\\ 
\text{and } \st(a_1\cdots a_{n-j})&=\st(a_{j+1}\cdots a_{n}). \label{eq_j}
\end{align}
Assume without loss of generality that $a_1<a_2$.  Equation~\eqref{eq_i} implies that $a_{k+1}<a_{k+2}$ and can be applied repeatedly to attain $a_{mk +1}<a_{mk +2}$ for any positive integer $m$ with $mk+2\leq n$.  Equation~\eqref{eq_j} can likewise be used to shift the indices in the inequality by multiples of $j$.  The equations are also capable of shifting the indices in the negative direction so that $a_\ell<a_{\ell+1}$ would imply $a_{\ell-j}<a_{\ell+1-j}$ if $\ell>j$.

Since it is assumed that $\gcd(k,j)=1$, for each $c\in \{2,\ldots, n-1\}$ there exists $d,e\in\mathbb{Z}$ such that $d\cdot k+e\cdot j=c-1$.  Therefore, we can achieve the inequality $a_{c}<a_{c+1}$ from $a_1<a_2$ by applying Equation~\eqref{eq_i} $d$ times and Equation~\eqref{eq_j} $e$ times.  If either of $d$ or $e$ are negative, the corresponding equation is applied each time to shift the indices in the negative direction.  Since $k+j<n$ we can order the applications of the equations to maintain that the first index of the inequality is always within $\{1,\ldots, n-1\}$.  Since we now have $a_c<a_{c+1}$ for any $c\in \{2,\ldots, n-1\}$, the vertex $\av$ must be the trivial vertex $\av=1 2\cdots n$, which is not within any $j$-cycle by Theorem~\ref{trivialVertices}$(b)$, contradicting that assumption.

\end{proof}

The following result shows that counting the number of closed $k$-walks is enough to count the number of $k$-cycles in $G(n)$ for any prime cycle length $k$.

\begin{corollary}\label{primek}
For a prime number $k$, the number of vertices that are within a $k$-cycle is given by $v_{n,k}=w_{n,k}-2$.
\end{corollary}
\begin{proof}
First, assume $\av$ is a nontrivial vertex, so $\av\neq 1\cdots n$ and $\av\neq n\cdots 1$.  If there exists a closed $k$-walk at vertex $\av$, it is either a $k$-cycle or it is a sequence of cycles of lengths $i_1,\cdots, i_\ell$ where $i_1+\cdots +i_\ell=k$.  By Theorem~\ref{combiningCycles}, $\av$ is only in cycles of two different lengths if those lengths are not relatively prime.  Thus each pair of lengths shares a common factor, resulting in some integer $d>1$ that divides each of $i_1,\ldots,i_\ell$.  Thus, $d$ must also divide their sum $k$, which is not possible since $k$ is prime.  Therefore, every closed $k$-walk that includes the nontrivial vertex $\av$ is in fact a $k$-cycle.  

In the case of $\av=1\cdots n$ or $n\cdots 1$, a closed $k$-walk must only be a repetition of the $1$-cycle since $\av$ is not included in any $j$-cycle for $1<j<n$ by Theorem~\ref{trivialVertices}$(b)$.  Since the trivial vertices are still included in $w_{n,k}$ in the count within Theorem~\ref{countingVertices2k} and the bound in Theorem~\ref{countingVertices2k3k}, we must subtract those two vertices from $w_{n,k}$.
\end{proof}

\section{Concluding Remarks}\label{remarks}

Our methods for counting the number of cycles in $G(n)$ fall short when it comes to finding a generalized method, though we were able to find the number of $2$-cycles in $G(n)$ and this same method could possibly work for finding the number of $3$-cycles in $G(n)$. We were able to establish a list of conditions for when a closed $k$-walk can and cannot exist. An intriguing result we found was Corollary~\ref{primek}. Combined with Theorem~\ref{countingVertices2k}, it provides an exact count for the number of vertices in $k$-cycles for an infinite set of pairs of $k$ and $n$ with $k\geq 3$, $k$ being prime, and $n\leq 2k$.  Furthermore, if we can pin down the number of closed $k$-walks in $G(n)$ when $n>2k$, then we can use this corollary to immediately find the number of cycles in a graph when $k$ is prime for any size of permutations.

There are several unanswered questions, some of which are shown below.

\begin{question} 
How many vertices are in $k$-cycles when $k$ is not prime?
\end{question}

\begin{question}
How many $k$-cycles (or closed $k$-walks) are there in $G(n)$ for $3\leq k< n$?
\end{question}

An interesting extension to this question which has received no attention are the following two questions when $k\geq n$.  Note that the enumerations of cycles for $G(n,312)$ in~\cite{EKS} also only counted cycles with length at most $n$.

\begin{question}
How many $k$-cycles (or closed $k$-walks) are there in $G(n)$ for $k\geq n\geq 3$?
\end{question}

It can be seen in Figure~\ref{G3} that there does exist a $4$-cycle $(132,213,231,312)$, so would it be possible to count the number of $k$-cycles in $G(n)$ when $k$ is larger than $n$? To this same point, it is quite a common problem to find a Hamilton cycle in graphs which leads us to our next question.  Note that a similar question was answered by Horan and Hurlbert~\cite{HoranHurlbert} for $s$-overlap cycles, but their work centered on $k$-permutations where $k<n$.

\begin{question}
Does there exist a Hamilton cycle in $G(n)$ for all $n$? If so, how many?
\end{question}

As described in~\cite{EKS}, when avoiding length $3$ patterns in $G(n)$, several patterns result in identical numbers of cycles.  The other distinct case that was left unanswered in their work involves avoiding the pattern $321$.

\begin{question}
Can our results for closed $k$-walks and cycles on the entire graph $G(n)$ be used to assist in determining the number of $k$-cycles (or closed $k$-walks) in $G(n,321)$ or in other subgraphs $G(n,\pi)$ where $\pi$ is a pattern of length at least $4$?
\end{question}


 \bibliographystyle{amsplain}
\bibliography{vdec}

\end{document}